\documentclass{amsart}

\usepackage{graphicx}

\newtheorem{theorem}{Theorem}[section]
\newtheorem{corollary}{Corollary}[theorem]
\newtheorem{lemma}[theorem]{Lemma}

\theoremstyle{definition}

\theoremstyle{remark}

\numberwithin{equation}{section}

\begin{document}
\

\title{Subtrees of a Given size in Schroeder Trees.}

\author{Anthony Van Duzer}
\address{Department of Mathematics, University of Florida, Gainesville, Florida 32601}
\curraddr{Department of Mathematics,
University of Florida, Gainesville, Florida 32611}
\email{avanduzer@ufl.edu}




\begin{abstract}
In this paper we find the generating function for the number of vertices which have $k$ elements in their subtree and use this generating function to calculate the probability that a vertex has a size $k$ subtree.  We also show how this same technique can be applied to calculate the probabilities for other trees and specifically apply it to Motzkin trees.
\end{abstract}

\maketitle

\section{Introduction}

\subsection{Schroeder Trees}
A Schroeder tree is a rooted plane tree on $n$ leaves where each non leaf vertex has at least two children.
\subsection{Purpose of paper}
In this paper we will find a bivariate generating function for the number of trees where the root has a subtree of size $k$ and from that we will obtain the generating function for the number of vertices that are $k$ protected.  Using that, we will calculate the probability a randomly selected vertex has a size $k$ subtree.  We will also show how this technique can be applied to calculate the probability that a vertex has some other property in a Schroeder tree and how we can also use it to calculate things about other trees.
\subsection{Subtree of size k}  We say that a vertex has a subtree of size $k$ if the subtree consisting of that vertex and all its descendents has $k$ elements.
\subsection {Number of Schroeder Trees}  The first step in doing any of these calculations will involve finding out exactly how many trees we have and once we know that we will also know exactly how many leaves we have.
\begin{theorem}
The generating function, $S(x)=\displaystyle\sum s_n x^n$, for the number of all Schroeder trees with n leaves is given by $S(x)=\dfrac{2x}{1+x+\sqrt{1-6x+x^2}}$ .  In this paper all generating functions will be referring to ordinary generating functions.
\end{theorem}
\begin{proof}
This comes from the relationship $S(x)=x+(S(x))^2+(S(x))^3+(S(x))^4+...$ which gives us $S(x)=x+\dfrac{(S(x))^2}{1-S(x)}$.  This relationship is based on the fact the root can be a leaf or the parent of at least two children.  If it is a leaf it contributes $x$ and if it is the parent of n children each child would be another Schroeder tree, we do not multiply by $x$  since we are counting by leaves and the root would not be a leaf.  Solving this equation we get the desired generating function.
\end{proof}
\newpage
\begin {corollary}
Let $l(n)$ be the number of leaves in all Schroeder trees of size n.  Then we have $l(n) \sim \dfrac{(1+\sqrt{2})(n-1)^{(\frac{-1}{2})}}{2^{\frac{7}{4}}\sqrt{\pi}}\left(3+\sqrt{8}\right)^{n-1}$.
\end{corollary}
\begin{proof}
We can extract the coefficent from the generating function and we get that $S(n)$ is asymptotically equal to $\dfrac{(1+\sqrt{2})(n-1)^{(\frac{-3}{2})}}{2^{\frac{7}{4}}\sqrt{\pi}}\left(3+\sqrt{8}\right)^{n-1}$, and there are precisely $n$ leaves on each Schroeder tree, so we simply multiply that number by $n$.
\end{proof}
\section{Number of Vertices} Before we can make any statements about the probability a randomly selected vertex has some interesting property we need to calculate the number of vertices.
\begin{theorem}
The generating function for the number of vertices in all Schroeder trees with  n leaves is given by $$V(x)=\dfrac{3-x+\sqrt{1-6x+x^2}}{4\sqrt{1-6x+x^2}}\left(\dfrac{2x}{1+x+\sqrt{1-6x+x^2}}\right).$$
\end {theorem}
\begin{proof}
We set up a bivariate generating function where $x $ counts the leaves and $y$ counts the vertices.  We call this $V(x,y)$ from there, if we differentiate with respect to $y$ and then set $y=1$ we get the univariate generating function for the number of all vertices in trees with $n$ leaves.  We have the relationship $V(x,y)=xy+y(V(x,y))^2+y(V(x,y))^3+y(V(x,y))^4+...$ which gives us $V(x,y)=xy+\dfrac{y(V(x,y))^2}{1-V(x,y)}$.  Solving for $V(x,y)$ we get $$V(x,y)=\dfrac{1+xy-\sqrt{(xy)^2+2xy+1-4xy(y+1)}}{2y+2}.$$  We then differentiate with respect to $y$ and substitute in $y=1$ to get the desired generating function.
\end{proof}
\begin {corollary}
Let $V(n)$ be the number of all vertices in all Schroeder trees with $n$ leaves.  We have $V(n) \sim \dfrac{(3+\sqrt{8})^n}{2^{\frac{9}{4}}\sqrt{\pi n}}$.
\end{corollary}
\begin{proof}
We can extract the coefficent from the previous generating function giving us these asymptotics.
\end{proof}
Using this we can calculate the probability that a randomly selected vertex in a randomly selected Schroeder tree is a leaf.
\begin {corollary}
The probability that a randomly selected vertex in a random Schroeder tree is a leaf approaches $\dfrac{1+\sqrt{2}}{\sqrt{2}(3+\sqrt{8})}\approx .293$ as the number of leaves goes to infinity.
\end{corollary}
\begin{proof}
We compare the number of leaves which was calculated in Section 1 to the number of vertices that was calculated in the previous theorem.
\end{proof}
This is also the probability that a vertex has subtree size 1 since every leaf has a subtree of size one and the only vertices with a size one subtree are those vertices which are leaves.
\section{Size $k$-subtrees}
\begin{lemma}
The generating function for the number of vertices which have size $k$ subtrees is $$(R_k(x))\dfrac{3-x+\sqrt{1-6x+x^2}}{4\sqrt{1-6x+x^2}},$$ where $R_k(x)$ is the generating function for the number of trees with $k$ vertices.
\end{lemma}
\begin{proof}
This follows from the relationship $$T_k(x)=R_k(x)+2T_k(x)S(x)+3T_k(x)(S(x))^2+4T_k(x)(S(x))^3+... \hspace{.1 in}.$$  The $R_k(x)$ accounts for all cases where the root has a subtree of size $k$.  After accounting for that we can cut off the root and that will break down our tree into a forest of some number of trees $m$.  We look at a specific subtree in that forest.  The number of vertices with a size $k$ subtree in that tree is given by $R_k(x).$  We now multiply that by the total number of configurations which is $S(x)$ for each of the remaining trees.  This means $T_k(x)S(x)^{m-1}$ gives the number of vertices that one of the $m$ subtrees contributes and then we multiply that by $m$ to get the total number of all of the vertices in the forest.  Looking at that we get 
\newline
$T_k(x)=R_k(x)+T_k(x)\dfrac{S(x)}{(1-S(x))^2}$.  Substituting in what we know we get the desired generating function.
\end{proof}
As an aside we could have also used this technique to find the generating function for the number of vertices and the number of leaves.  In that case for leaves we would have $$L(x)=x\dfrac{3-x+\sqrt{1-6x+x^2}}{4\sqrt{1-6x+x^2}},$$ since the only way the root can be a leave is if it is the only vertex and that tree has generating function $x$.  Similarly the number of vertices is given by $$S(X)\dfrac{3-x+\sqrt{1-6x+x^2}}{4\sqrt{1-6x+x^2}},$$ since the root is always a vertex.
\newline
Now all we need to do is find $R_k(x)$.  We can do this using bivariate generating functions.
\begin{lemma}
Let $R(x,y)$ be the bivariate generating function where $x$ is indexed by the number of leaves in the tree and $y$ is by the number of total vertices in the tree.  With that we have $$R(x,y)=\dfrac{1+xy-\sqrt{(xy)^2+2xy+1-4xy(y+1)}}{2y+2}.$$
\end{lemma}
\begin{proof}
This is the bivariate generating function we got when trying to find the number of vertices and is calculated in exactly the same way.
\end{proof}
Now that we have that generating function $R_k(x)$ is simply the coefficent of $y^k$ in this expression which will be a generating function in $x$.
\newpage
\begin {lemma} (Bender's Lemma) \cite{D}
\newline
Take generating functions $A(x)=\sum a_nx^n$ and $B(x)=\sum b_n x^n$ with radius of convergence $\alpha >\beta \geq 0$ where $\alpha$ goes with $A(x)$ and $\beta$ goes with $B(x)$.  If $\frac{b_{n-1}}{b_n}$ approaches a limit b as n approaches infinity and $A(b) \neq 0$ then $c_n \sim A(b)b_n$ where $\sum c_nx^n$=A(x)B(x).
\end {lemma}
We can use the probabilities we already have, the generating functions, and Bender's lemma to calculate the probabilities and doing that we get.
\begin{center}
 \begin{tabular}{|c | c| c|} 
 \hline
k & $R_k(x)$ & Probability the subtree is size$k$ $\approx$\\ [.5ex] 
 \hline
 1 & $x$ & .2929\\ 
 \hline
 2 & 0 & 0 \\
 \hline
3 & $x^2$ & 0.0503 \\
\hline
4 & $x^3$ & 0.0086\\
\hline
5 & $2x^3+x^4$ & 0.0187\\
\hline
6 & $5x^4+x^5$ & 0.0076\\
\hline
7 & $5x^4+9x^5+x^6$ & 0.0097\\
\hline
\end{tabular}
\end{center}
\section{Balanced and $k$-protected}
This techinque can also be used to find the probability a vertex has one of several other properties.  The only thing we need to find is the generating function for the number of trees where the root has that property.  As an example consider the probability a vertex is balanced and rank $k$.
\subsection {Rank k} A vertex is of rank $k$ if the shortest path from that vertex to a leave travelling strictly through descendents is of length $k$.
\subsection{Balanced} A vertex is balanced if the shortest and longest path to a leaf is of the exact same length.  Putting these two definitions together a vertex is balanced of rank $k$ if all paths travelling from the vertex to a leaf through descendents are of length $k$.
\newline
The first thing we need to do is find the generating function for the root being balanced and rank $k$. 
\begin{theorem} Let $B_k^*(x)$ be the generating functions for where the root is balanced and rank $k$. We have $B_k^*(x)=\dfrac{(B_{k-1}^*(x))^2}{1-B_{k-1}^*(x)}$.
\end{theorem}
\begin{proof}
For any vertex, including the root to be balanced and rank $k$ all of its children must be balanced of rank $k-1$.  From this we know we need to construct a tree where the root is balanced and rank $k$ we need to construct a sequence of at least 2 trees where the root of each of these trees is balanced of rank $k-1$.  We can do this in $\dfrac{(B_{k-1}^*(x))^2}{1-B_{k-1}^*(x)}$ ways.
\end{proof}
\begin{theorem}
Let $B_k(x)$ be the generating function for all vertices who are balanced and rank $k$.  We have $B_k(x)=B_k^*(x)\dfrac{3-x+\sqrt{1-6x+x^2}}{4\sqrt{1-6x+x^2}}$.
\end{theorem}
\begin{proof}
This is similar to the situation for size $k$ subtree and the generating function is proved in the exact same way.
\end{proof}
We can combine the previous theorems, the fact that $B_0^*(x)=x$, and Bender's lemma to construct the following table of probabilities.
\begin{center}
 \begin{tabular}{|c | c| c|} 
 \hline
k & $B_k^*(x)$ & Probability the vertex is balanced and $k$ protected $\approx$\\ [1ex] 
 \hline
 0 & $x$ & .2929\\ 
 \hline
 1 & $\dfrac{x^2}{1-x}$ & .0607 \\
 \hline
2 & $\dfrac{x^4}{1-2x+x^3}$ & .0022 \\
\hline
3 & $\dfrac{x^8}{1-4x+4x^2+2x^3-5x^4+2x^5+x^6-x^7}$ &0.000003\\
\hline
\end{tabular}
\end{center}
\section{Further Directions}
The technique used in this paper to calculate the probability a vertex has a subtree of size $k$ very easily generalizes to other types of trees.  The only thing we need to find is the generating function for leaves and the generating function for the number of trees where the root has a subtree of size $k$.  This makes it particular well suited to problems where rather than counting the tree by the number of leaves we are counting them by the number of vertices since in that case the generating function for the number of trees where the root has a subtree of size $k$ is simply $T_k x^k$ where $T_k$ is the number of trees with $k$ vertices.  As an example we can look at Motzkin trees.
\subsection{Motzkin Trees}
So for Motzkin Trees we have the generating function for the number of vertices with subtrees of size $k$  in trees with $n$ vertices is given by $\dfrac{M_k x^k}{\sqrt{1-2x-3x^2}}$.  From this, Bender's lemma, and previously obtained results regarding Motzkin trees we know that the probability a vertex has a subtree of size $k$ is simply $M_k\left(\frac{1}{3}\right)^k$.  If we sum this over all $k$ we get 1 which shows us that this is a tight statistic for Motzkin trees.
\section{Open questions}
We know that for the Motzkin trees size $k$ subtrees is a tight statistic.  We also know that there exist classes of trees for which the statistic is not tight the most obvious being the class of trees for which each parent is allowed only 1 child.  In such a tree the asymptotic probability that an arbitrary vertex has a size $k$ subtree is 0 for all $k$ as the number of vertices goes to $\infty$.  Two interesting questions we could ask are is this stastic tight for Schroeder trees and in general is there any way to classify trees where this statistic is tight.  Also we could ask are there any non degenerate tree families for which this statistic is not tight.
\vspace {.1 in}
\newline
Another interesting question would be about the sequence of probabilities for Schroeder trees themselves.  We can see with the initial few numbers that the probabilities are not strictly decreasing.  This makes sense because of the restriction of needing every parent to have at least two children.  It would be interesting to prove either the sequence is eventually monotonically decreasing or it continues to oscilate forever.    It is related to the problem of compositions of numbers without using the number one since we can always see how many vertices the tree has by adding up the number of children each vertice has but we can never have a vertex with only one child.


\begin{thebibliography}{10}
\bibitem {A}Gi-Sang Cheon, Louis W. Shapiro \textit{Protected points in ordered trees},
Applied Mathematics Letters vol. 21 no. 5 pp. 516 - 520
\bibitem {B}Toufik Mansour, \textit{Protected points in k-ary trees},
Applied Mathematics Letters vol. 24 number 4 pp.478 - 480
\bibitem {C} Mikl\'{o}s B\'{o}na, \textit{On the number of vertices of each rank in phylogenetic trees and their generalizations},
Discrete Mathematics \& Theoretical Computer Science, Vol. 18 no. 3, Combinatorics (April 11, 2016) dmtcs:1431
\bibitem {D} Philippe Flajolet, Robert Sedgewick \textit{Analytic Combinatorics}, Cambridge University Press, Cambridge (2009)
\end{thebibliography}
\end{document}